\documentclass[10 pt, leqno]{article}

\date{}
\usepackage{amssymb, amsbsy, amsmath, amsfonts, amssymb, amscd, mathrsfs, amsthm}

\usepackage[english]{babel}
\usepackage[T1]{fontenc}
\usepackage{indentfirst}
\usepackage{color}
\makeatletter
\@addtoreset{equation}{section}
\makeatother

\newtheorem{statement}{}[section]
\newtheorem{theorem}[statement]{Theorem}
\newtheorem{lemma}[statement]{Lemma}

\newcommand\C{\mathbb C}

\newcommand\R{\mathbb R}
\newcommand\T{\mathbb T}
\newcommand\D{\mathbb D}
\newcommand\Z{\mathbb Z}
\newcommand\e{{\rm e}}

\newcommand\eps{\varepsilon}
\newcommand\ind{1 \kern - 0.28 em {\rm I}}
\newcommand\dis{\displaystyle}
\renewcommand \Re{{\mathfrak R}{\rm e}\,}
\renewcommand \Im{{\mathfrak I}{\rm m}\,}
\newcommand\converge{\mathop{\longrightarrow}\limits}

\let\amphi=\phi
\let\phi=\varphi

\newcommand\tq{\, ; \ }

\title{\bf On composition operators on the Wiener algebra of Dirichlet series}
\author{\it Daniel~Li, \it Herv\'e~Queff\'elec, Luis~Rodr{\'\i}guez-Piazza}

\date{\footnotesize \today}

\begin{document}

\maketitle

\noindent {\bf Abstract.} We show that the symbol of a bounded composition operator on the Wiener algebra of Dirichlet series does not need to belong to this algebra. 
Our example even gives an absolutely summing (hence compact) composition operator. 

\medskip

\smallskip


%%%%%%%%%%%%%%%%%%%%%%%%%%%%%%%%%%%%%%%%%%%%%%%%%%%%%%%%%%%%%%%%%%%%%%%%%%%
\section {Introduction} 

In \cite{BFLQ} (see also \cite{FLQ}), composition operators on the Wiener algebra ${\mathcal A}^+$ of all absolutely convergent Dirichlet series were studied. 
\smallskip

Recall that ${\mathcal A}^+$ is the space of all analytic maps $f \colon \C_{\, 0} \to \C$ which can be written 
\begin{displaymath}
f (s) = \sum_{n = 1}^\infty a_n n^{- s} \qquad \text{with} \qquad  
\| f \|_{{\mathcal A}^+} := \sum_{n= 1}^\infty |a_n| < + \infty \, , 
\end{displaymath}
where, for $\theta \in \R$, we note $\C_{\, \theta} = \{ z \in \C \tq \Re z >\theta \}$. 
If $\amphi \colon \C_{\, 0} \to \C_{\, 0}$ is an analytic function, the composition operator $C_\amphi \colon {\mathcal A}^+ \to {\mathcal A}^+$ 
of symbol $\amphi$ on this space is defined as $C_\amphi (f) = f \circ \amphi$. Gordon and Hedenmalm, for the Hilbert space ${\mathscr H}^2$, showed in 
\cite{GH} that such a symbol has necessarily the form 
\begin{equation} \label{condition GH}
\amphi (s) = c_0 \, s + \phi (s) \, , 
\end{equation}
where $c_0 \geq 0$ is an integer and $\phi$ is a convergent Dirichlet series with values in $\C_{\, 0}$, that is 
$\phi \colon \C_{\, 0} \to \C_{\, 0}$ is an analytic function which can be written $\phi (s) = \sum_{n = 1}^\infty c_n n^{- s}$ for $\Re s$ large enough. 
Moreover, this Dirichlet series is uniformly convergent in $\C_{\, \eps}$ for all $\eps > 0$ (\cite[pages 1625--1626 and Theorem~3.1]{Q-S}; see also 
\cite[Theorem~8.4.1, page 245]{Herve-Martine}). 

It is shown in \cite[Theorem~2.3]{BFLQ} that $C_\amphi$ is bounded on ${\mathcal A}^+$ if and only if 
$\sup_{N \geq 1} \| N^{- \amphi} \|_{{\mathcal A}^+} < + \infty$, and that it is compact if and only if 
$ \| N^{- \amphi} \|_{{\mathcal A}^+} \converge_{N \to \infty} 0$. Note that it is actually proved in \cite[Theorem~4]{GH} that 
if $N^{- \amphi}$ is a Dirichlet series for all $N \geq 1$, then $\amphi$ as necessarily the form \eqref{condition GH}. Then 
$\| N^{- \amphi} \|_{{\mathcal A}^+} = \| N^{- \phi} \|_{{\mathcal A}^+}$, so $c_0$ plays no role, so we assume in the sequel that $c_0 = 0$. 
\smallskip

When $X$ is a Banach space of analytic functions that contains the identity map $u \colon z \mapsto z$, and $C_\amphi \colon X \to X$ is a composition operator, 
then $\amphi = C_\amphi (u)$ belongs to $X$. For $X = {\mathcal A}^+$, it is not the case, so it is natural to ask if $\phi \in {\mathcal A}^+$ when 
$C_\phi \colon {\mathcal A}^+ \to {\mathcal A}^+$ is a bounded composition operator. The object of this short note is to give a negative answer 
(Theorem~\ref{main theorem}). 
\smallskip

Let us point out that it is proved in \cite[Proposition~2.9]{BFLQ} that $\phi \in {\mathcal A}^+$ does not suffice to have a bounded composition operator on 
${\mathcal A}^+$; the symbol is even a Dirichlet polynomial 
\begin{displaymath}
\phi (s) = c_1 + c_r r^{-s} + c_{r^2} r^{- 2 s} 
\end{displaymath}
where $r \geq 2$ is an integer and $c_r, c_{r^2} > 0$. For such a Dirichlet polynomial, it is proved that 
$C_\phi$ is not bounded if $\Re c_1 < \frac{(c_r)^2}{8 \, c_{r^2}}$ and $c_r \leq 4 \, c_{r^2}$ (for example, $c_r = 4$, $c_{r^2} = 1$ and $\Re c_1 < 2$). 

%%%%%%%%%%%%%%%%%%%%%%%%%%%%%%%%%%%%%%%%%%%%%%%%%%%%%%%%%%%%%%%%%%%%%%%%%%%%%
\section{Main result} 

Recall that $\phi \colon \C_{\, 0} \to \C$ is a convergent Dirichlet series, if $\phi$ is analytic on $\C_{\, 0}$ and we can write 
$\phi (s) = \sum_{n = 1}^\infty a_n n^{ - s}$ for $\Re s$ large enough.  

\begin{theorem} \label{main theorem}
There exists a convergent Dirichlet series $\phi$ inducing a bounded composition operator $C_\phi \colon {\mathcal A}^+ \to {\mathcal A}^+$, but 
such that $\phi \notin {\mathcal A}^+$. Moreover, $\phi \in {\mathscr H}^p$ for all $p < \infty$ and $C_\phi$ is compact and absolutely summing. 
\end{theorem}

Let us recall the definition of the Hardy space ${\mathscr H}^p$ of Dirichlet series, following \cite{Bayart}. That uses the Bohr representation of 
Dirichlet series. Let $(p_j)_{j \geq 1}$ be the increasing sequence of all the prime numbers (so $p_1 = 2$, $p_2 = 3$, $p_3 = 5$, and so on). If 
$n = p_1^{\alpha_1} p_2^{\alpha_2} \cdots p_r^{\alpha_r}$ is the decomposition of the integer $n$ in prime factors, to the 
Dirichlet series $\phi (s) = \sum_{n = 1}^\infty a_n n^{- s}$ is associated the Taylor series 
$(\Delta \phi) (z) = \sum_{{\boldsymbol \alpha}} a_n z_1^{\alpha_1} z_2^{\alpha_2} \cdots z_r^{\alpha_r}$, where 
${\boldsymbol \alpha} = (\alpha_1, \alpha_2, \ldots, \alpha_r, 0, 0, \ldots)$. Due to Kronecker's theorem, $\phi$ is bounded if and only if $\Delta \phi$ 
is bounded, and $\| \phi \|_\infty = \| \Delta \phi \|_\infty$. The Hardy space ${\mathscr H}^p$ is the space of all convergent Dirichlet 
series $\phi$ for which $\Delta \phi$ belongs to the Hardy space $H^p (\T^\infty)$, with the norm 
$\| \phi \|_{{\mathscr H}^p} = \| \Delta \phi \|_{H^p}$. 

Note that ${\mathcal A}^+$ is isometrically isomorphic, by this map $\Delta$, to the Wiener algebra $A^+ (\T^\infty)$. 
\smallskip

Let us also recall that a bounded linear map $u \colon X \to Y$ between two Banach spaces $X$ and $Y$ is $r$-summing ($1 \leq r < \infty$) if there is a 
positive constant $K$ such that 
\begin{displaymath}
\bigg( \sum_{k = 1}^n \| u (x_k) \|^r \bigg)^{1 / r} \leq K \, \sup_{\xi \in B_{X^\ast}} \bigg( \sum_{k = 1}^n | \xi (x_k) |^r \bigg)^{1 / r} 
\end{displaymath}
for all $x_1, \ldots, x_n \in X$, $n \geq 1$, and where $B_{X^\ast}$ is the unit ball of $X^\ast$. For $r = 1$, these operators are also said absolutely 
summing. 

\begin{proof} [Proof of the theorem] 
We are going to take a symbol $\phi$ of the form $\phi (s) = \sum_{k = 1}^\infty c_k \, 2^{- k s} = f (2^{- s})$, where $f \colon \D \to \C_{\, 0}$ is an 
analytic function such that $\sup_{N \geq 1} \| N^{- f} \|_{A^+ (\D)} < + \infty$, but $f \notin H^\infty$. 

Recall that $A^+ = A^+ (\D)$ is the space of all analytic functions $u \colon \D \to \C$ such that $u (z) = \sum_{n = 0}^\infty a_n z^n$, with 
$\| u \|_{A^+ (\D)} := \sum_{n = 0}^\infty | a_n | < + \infty$. 
\smallskip

We choose for $f$ a conformal map sending the unit disk $\D$ onto the half-strip 
\begin{displaymath}
R = \{ z \in \C \tq \Re z > 1 \text{ and } |\Im z | < \pi \} \, . 
\end{displaymath}
Explicitly, we take $f = \tau_1 \circ L \circ h \circ c \circ T$, where 
\begin{displaymath}
\begin{array}{ll}
\dis T (z) = \frac{1 + z}{1 - z} \, ; & \dis c (z) = \e^{i \pi / 4} \sqrt{z} \, ; \smallskip \\ 
\dis h (z) = \frac{i z + 1}{z + i} \, ; & \dis L (z) = - 2 \log z \, ;
\end{array}
\end{displaymath}
and $\tau_1 (z) = z + 1$. $T$ maps the unit disk $\D$ onto the right-half plane; then $c$ sends the right-half plane onto the first quadrant; $h$ the first 
quadrant onto the right-half of $\D$; $L$ this right-half of $\D$ onto the half-strip $\{ | \Im z | < \pi , \Re z > 0 \}$, and finally the translation $\tau_1$ 
sends this half-strip onto the half-strip $R$.

This map is clearly not in $H^\infty$, but, for every $\beta \in (0, \pi / 2)$, there is a positive constant $C_\beta$ such that $R + C_\beta$ is contained 
in the angular sector of vertex $0$ and of opening $\beta$; it follows (see \cite[Theorem~3.2]{Duren}) that $f + C_\beta \in H^p$ for all $p < {\pi / \beta}$; 
so $f \in H^p$ for all $p < \infty$. We can also see that 
\begin{displaymath}
f (\e^{i t}) = \alpha \, \log | \, i - \e^{i t} | + g (t) \, , 
\end{displaymath}
with $g \in L^\infty$ and $\alpha$ a constant, so that $f \in H^{\Psi_1}$, the Hardy-Orlicz space attached to the Orlicz function 
$\Psi_1 (x) = \e^x - 1$, and 
\begin{displaymath}
\| f \|_p = O \, (p) 
\end{displaymath}
as $p$ goes to infinity. 

Since $f \notin H^\infty$, we a fortiori have $f \notin A^+$, so $\phi \notin {\mathcal A}^+$. 
However $\phi \in {\mathscr H}^p$ for all $p < \infty$ since $f \in H^p$ for these values of $p$. 
\smallskip

We now have to show that $N^{- \phi} \in {\mathcal A}^+$, i.e. $N^{- f} \in A^+$, for all $N \geq 1$. This is clear for $N = 1$. For $N \geq 2$, we have 
$N^{- f} = \exp \, (- f \log N)$, and the range of $f \log N$ is the half-strip
\begin{displaymath}
R_N = \{ z \in \C \tq \Re z > \log N \text{ and } | \Im z | < \pi \log N \} \, .
\end{displaymath}
Under the exponential map $\e^{- z}$, $\partial R_N$ is transformed as follows: 

1) the vertical segment $[\log N - i \pi \log N, \log N + i \pi \log N]$ is sent onto the circle of center $0$ and radius $1 / N$, browsed $\log N $ times; 

2) the half-line $\{ t \, \log N + i \pi \log N \tq t \geq 1 \}$ is one-to-one mapped onto the radius $(0, \e^{- i \pi \log N} / N ]$; 

3) the half-line $\{ t \, \log N - i \pi \log N \tq t \geq 1 \}$ is one-to-one mapped onto the radius $(0, \e^{i \pi \log N} / N ]$. 

Hence, if $F_N = N^{- f}$, we have 
\begin{displaymath}
\int_0^{2 \pi} | F_N ' (\e^{i t}) | \, dt = \frac{2}{N} + 2 \pi \, \frac{\log N}{N} < + \infty \, , 
\end{displaymath}
so $(N^{- f})' \in H^1$.
By Hardy's inequality 
(see \cite[Corollary page~48]{Duren}), it follows that $N^{- f} \in A^+$, and there exists a positive constant $C$ such that 
$\| N^{- f} \|_{A^+} \leq C \, \log N / N$. In particular, $\| N^{ - \phi} \|_{{\mathcal A}^+} = \| N^{ - f} \|_{A^+} \converge_{N \to \infty} 0$, 
so $C_\phi \colon {\mathcal A}^+ \to {\mathcal A}^+$ is compact, by \cite[Theorem~2.3]{BFLQ}. 
\smallskip

To end the proof, remark that, writing $u_N (s) = N^{- s}$, we have 
\begin{displaymath}
\sum_{N = 2}^\infty \| C_\phi (u_N) \|_{{\mathcal A}^+}^2 = \sum_{N = 2}^\infty \| N^{ - \phi} \|_{{\mathcal A}^+}^2 
\leq C^2 \sum_{N = 2}^\infty \frac{(\log N)^2}{N^2} < + \infty \, . 
\end{displaymath}
Hence, using the Cauchy-Schwarz inequality, we can define a bounded linear operator $S \colon \ell_2 \to {\mathcal A}^+$ by sending the $N$-th vector $e_N$ 
of the canonical basis of $\ell_2$ to $N^{- \phi}$, and we have the factorization $C_\phi = S T$, where $T \colon {\mathcal A}^+ \to \ell_2$ is defined by setting 
$T (u_N) = e_N$. But ${\mathcal A}^+$ is isometrically isomorphic to $\ell_1$, and the canonical injection from $\ell_1$ to $\ell_2$ is $1$-summing (this 
was first remarked by Pietsch \cite[\S\,1, Satz~10]{Pietsch}, and it is a particular case of the Grothendieck theorem). Let us recall why that holds. 
To each $(\alpha_k)_{k \geq 1} \in \ell_1$, we associate the $L^\infty$ function 
$\sum_{k = 1}^\infty a_k r_k$, where $(r_k)_{k \geq 1}$ is the sequence of the Rademacher functions on $[0, 1]$; the canonical injection from 
$L^\infty (0, 1)$ into $L^1 (0, 1)$ is absolutely summing (see \cite[top of page 11]{Pisier}, or \cite[Remark~8.2.9]{Albiac-Kalton}) and, 
by Khintchin's inequalities (see \cite[Chapitre~0, Th\'eor\`eme~IV.1]{LQ}, or \cite[Chapter~1, Theorem~IV.1]{LQ-bis}), the $L^1$-norm of 
$\sum_{k = 1}^\infty \alpha_k r_k$ is equivalent to $\big( \sum_{k = 1}^\infty |\alpha_k|^2 \big)^{1 / 2}$. 
\smallskip

It follows that $C_\phi$ is $1$-summing. 
\end{proof}

Note that, since ${\mathcal A}^+ \cong \ell_1$ has the Schur property, and since every $q$-summing operator is weakly compact, 
every $q$-summing operator into ${\mathcal A}^+$ is compact. 
\medskip

A slight modification of the proof gives a variant of Theorem~\ref{main theorem}.

\begin{theorem}
For every $p \in (1, \infty)$, there exists a convergent Dirichlet series $\phi$ such that $\phi \in {\mathscr H}^q$ for all $q < p$, but $\phi \notin {\mathscr H}^p$, 
and such that $\phi$ induces a bounded composition operator $C_\phi \colon {\mathcal A}^+ \to {\mathcal A}^+$. Moreover, $C_\phi$ is compact and is 
absolutely summing.  
\end{theorem} 
\begin{proof} 
We replace the conformal map $f$ of Theorem~\ref{main theorem} by a conformal map $f$ from $\D$ onto the intersection of the angular sector 
\hbox{$\{ z \in \C_{\, 0} \tq | \arg z | < \pi / 2 p \}$} with the half-plane $\C_{\, 1}$. We have $f \notin H^p$ though $f \in H^q$ for all $q < p$ 
(see \cite[top of page 237]{Hansen}). 
We set $\phi (s) = f (2^{- s})$ for $\Re s > 0$. We have $\phi \in {\mathscr H}^q$ for all $q < p$, but $\phi \notin {\mathscr H}^p$. 

For all $N \geq 1$, we have $N^{- f} \in A^+$. This is clear for $N = 1$. For $N \geq 2$, let $\beta = \pi / 2 p$ and 
$\gamma_{\pm} (t) = \exp ( - \, \e^{\pm i \beta} t)$, with $t \geq \log N / \cos \beta$, then the boundary of the range of $F_N = N^{- f}$ is the union of 
$\gamma_+$ and $\gamma_-$, and of the circle of radius $1 / N$ browsed $(1 / \pi) \, (\tan \beta) \, \log N$ times. 
Since 
\begin{displaymath}
\int_{\log N / \cos \beta}^{+ \infty} | \gamma_\pm ' (t) | \, dt = \int_{\log N / \cos \beta}^{+ \infty} \e^{- ( \cos \beta) t} \, dt 
= \frac{1}{\cos \beta} \, \frac{1}{N} \, \raise 1 pt \hbox{,}
\end{displaymath}
we get that 
\begin{displaymath}
\int_0^{2 \pi} |F_N ' (\e^{i t}) | \, dt = \frac{2}{\cos \beta} \, \frac{1}{N} + \frac{\tan \beta}{\pi} \, \frac{\log N}{N} < + \infty \, ,
\end{displaymath}
so $F_N ' \in H^1$ and $F_N = N^{- f} \in A^+$. 
\smallskip

Moreover, $\| N^{- \phi} \|_{{\mathcal A}^+} = \| N^{- f} \|_{A^+} \lesssim \log N / N \converge_{N \to \infty} 0$, so $C_\phi$ is compact on 
${\mathcal A}^+$. 
\smallskip

Since $\sum_{N = 1}^\infty \| N^{- \phi} \|_{{\mathcal A}^+}^2 < + \infty$, we get, as in the proof of Theorem~\ref{main theorem}, that 
$C_\phi$ is $1$-summing.
\end{proof}
%

%%%%%%%%%%%%%%%%%%%%%%%%%
\section{Another result} 

Let us remark that the example of \cite[Proposition~2.9]{BFLQ} quoted in the Introduction is a Dirichlet polynomial $\phi$ such that 
$N^{- \phi} \in {\mathcal A}^+$ for all $N \geq 1$, though the associated composition operator $C_\phi$ is not bounded from ${\mathcal A}^+$ into itself. 

\begin{theorem} \label{limit theorem}
For any non-negative number $A \in \R_+$, there exists a convergent Dirichlet series $\phi$ such that $\phi (\C_{\, 0}) \subseteq \C_{\, A}$, but such that, for any 
$N \geq 2$, we have $N^{- \phi} \notin {\mathcal A}^+$. 

In particular, the composition operator $C_\phi$ is not bounded from ${\mathcal A}^+$ into itself. 
\end{theorem}

That will follow from the following result. 

\begin{lemma} \label{limit lemma}
Let $N \geq 2$ and let $\phi \colon \C_{\, 0} \to \C$ be an analytic function such that $N^{- \phi} \in {\mathcal A}^+$. Then, for every $a \in \R$, either $\phi (s)$ 
has a limit as $s$ tends to $i a$, or $\Re \phi (s)$ tends to $+ \infty$ as $s$ tends to $i a$. 
\end{lemma}
\begin{proof}
Since $N^{- \phi}$ belongs to ${\mathcal A}^+$, it is continuous on $\overline{\C_{\, 0}^{\phantom{l}}}$; hence it has limits at every point $i a \in i \R$. If this limit 
is $0$, that means that $\Re \phi (s) \converge_{s \to i a} + \infty$. If not, we have $\lim_{s \to i a} N^{- \phi (s)} = c \neq 0$. Therefore, if $r < |c |$, there is some 
open disk $V$ centered at $i a$ such that $N^{- \phi (s)} \in D (c, r)$ when $s \in V$. Let $F$ be a determination of the logarithm in $D (c, r)$. Then 
\begin{displaymath}
\psi (s) := F [ N^{- \phi (s)}] \converge_{s \to i a} F (c) \, . 
\end{displaymath}
Since 
\begin{displaymath}
\exp [ - \phi (s) \, \log N ] = N^{- \phi (s)} = \exp [\psi (s)] \, , 
\end{displaymath}
there exists $k = k (s) \in \Z$ such that $\psi (s) = - \phi (s) \, \log N + 2 k (s) \pi i$. But $\phi$ and $\psi$ are continuous on $V \cap \C_{\, 0}$; it follows that 
$k (s)$ is constant. Therefore 
\begin{displaymath}
\phi (s) = - \psi (s) + 2 k \pi i \converge_{s \to i a} - F (c) + 2 k \pi i \, . \qedhere
\end{displaymath}
\end{proof}
\begin{proof} [Proof of Theorem~\ref{limit theorem}]
Let 
\begin{displaymath}
\phi (s) =A + 1 + \exp \bigg( - \frac{1 + 2^{- s}}{1 - 2^{- s}} \bigg) \, \cdot
\end{displaymath}
Then $\phi$ is a convergent Dirichlet series and maps $\C_{\, 0}$ into $\C_{\, A}$. However, $N^{- \phi} \notin {\mathcal A}^+$ because $\phi$ does not have a limit 
as $s$ goes to $0$. 
\end{proof}
%

%%%%%%%%%%%%%%%%%%%%%%%%%%%%%%%%%%%%%%%%%%%%%%%%%%%%%%%%%%%%%%%%%%%%%%%%%%%%%
\bigskip

\noindent{\bf Acknowledgement.} This work was begun in june 2022, during an invitation of the two first-named authors by the IMUS of the Sevilla university. 
It is a great pleasure to us to thanks all people who made this stay possible and very pleasant, and especially Manuel Contreras. It has been completed during a 
visit of the third-named author in the Universit\'e de Lille and Universit\'e d'Artois in september 2022. 
\smallskip 

It was partially supported by the project PGC2018-094215-B-I00 (Spanish Ministerio de Ciencia, Innovaci\'on y Universidades, and FEDER funds). 
 
\goodbreak
%%%%%%%%%%%%%%%%%%%%%%%%%%%%%%%%%%%%%%%%%%%%%%%%%%%%%%%%%%%%%%%%%%%%%%%%%%%%%%

 %%%%%%%%%%%%%%%%%%%%%%%%%%%%%%%%%%%%%%%%%%%%%%%%%%%%%%%%%%
\smallskip\goodbreak

{\footnotesize
Daniel Li \\ 
Univ. Artois, UR~2462, Laboratoire de Math\'ematiques de Lens (LML), F-62\kern 1mm 300 LENS, FRANCE \\
daniel.li@univ-artois.fr
\smallskip

Herv\'e Queff\'elec \\
Univ. Lille Nord de France, USTL,  
Laboratoire Paul Painlev\'e U.M.R. CNRS 8524, F-59\kern 1mm 655 VILLENEUVE D'ASCQ Cedex, FRANCE \\
Herve.Queffelec@univ-lille.fr
\smallskip
 
Luis Rodr{\'\i}guez-Piazza \\
Universidad de Sevilla, Facultad de Matem\'aticas, Departamento de An\'alisis Matem\'atico \& IMUS,  
Calle Tarfia s/n  
41\kern 1mm 012 SEVILLA, SPAIN \\
piazza@us.es
}

\end{document}